\documentclass[letterpaper, 10 pt, conference]{ieeeconf}  

\IEEEoverridecommandlockouts                              
\IEEEoverridecommandlockouts                              
\usepackage{times} 
\usepackage{amsmath} 
\usepackage{amssymb}  

\usepackage{graphicx}  
\usepackage{amsmath}   
\usepackage{array}
\usepackage{times}
\usepackage{comment}
\usepackage{latexsym,theorem}
\usepackage{amsmath,amsfonts,amssymb,amsbsy}
\usepackage[ruled]{algorithm2e}

\newcommand{\Nset}{\mathbb{N}}
\newcommand{\Rset}{\mathbb{R}}
\newcommand{\Uset}{\mathbb{U}}
\newcommand{\Xset}{\mathbb{X}}

\newcommand{\cK}{\mathcal{K}}

\newcommand{\cL}{\mathcal{L}}

\newcommand{\bu}{\mathbf{u}}
\newcommand{\bx}{\mathbf{x}}

\newcommand{\diag}{\operatorname{diag}}

\usepackage{epsfig} 
\usepackage{amsmath} 
\usepackage{amssymb}  

{\theorembodyfont{\slshape}\newtheorem{theorem}{Theorem}[section]}
{\theorembodyfont{\slshape}\newtheorem{proposition}[theorem]{Proposition}}
{\theorembodyfont{\slshape}}
{\theorembodyfont{\slshape}}
{\theorembodyfont{\slshape}\newtheorem{corollary}[theorem]{Corollary}}
{\theorembodyfont{\upshape}\newtheorem{definition}[theorem]{Definition}}
{\theorembodyfont{\upshape}\newtheorem{problem}[theorem]{Problem}}
{\theorembodyfont{\upshape}\newtheorem{remark}[theorem]{Remark}}
{\theorembodyfont{\upshape}\newtheorem{assumption}[theorem]{Assumption}}

\title{\LARGE \bf
Stabilization of discrete--time nonlinear systems\\ based on control dissipation functions
}

\author{Mircea Lazar 
\thanks{M. Lazar is with the Department of Electrical Engineering, Eindhoven University of Technology, The Netherlands, E-mail:
        {\tt\small m.lazar@tue.nl}.}%
}

\begin{document}

\maketitle
\thispagestyle{empty}
\pagestyle{empty}

\begin{abstract}
This paper develops a dissipativity--based framework for synthesis of stabilizing controllers for discrete--time nonlinear systems subject to state/input constraints. Firstly, we revisit dissipation inequalities for discrete--time nonlinear systems that involve general storage and supply functions. We prove that positive definite storage functions and cyclically negative supply functions yield asymptotic Lyapunov stability. Secondly, we define control dissipation functions (CDFs) and we construct stabilizing receding horizon controllers by minimizing CDFs subject to a cyclically negative supply condition. The effectiveness of stabilizing controllers based on CDFs is demonstrated on stabilization of interconnected synchronous generators with nonlinear coupling and state/input constraints.
\end{abstract}

\section{Introduction}
\label{sec1}
The problem considered in this paper is stabilizing controller synthesis for discrete--time nonlinear systems that are possibly subject to constraints. The standard approach to this problem relies on control Lyapunov functions (CLFs) \cite{Artstein83, Sontag83, Kellett04SCL}, which provide a systematic framework for synthesis of stabilizing controllers. However, the construction of CLFs for general nonlinear systems is a challenging problem. Ideally, it would be desirable to use arbitrary positive definite functions as candidate CLFs. Indeed, since Sontag's universal formula \cite{Sontag83}, the most common approach in the literature assumes a quadratic positive definite function as a candidate CLF for the original nonlinear system at the cost of conservatism (e.g, computed as a local CLF for a linearization of the nonlinear dynamics).

Several approaches to stabilizing controller synthesis for discrete--time nonlinear systems based on CLFs have been explored, many of which have been developed within the field of nonlinear model predictive control (MPC). Therein, CLFs have been be used to design terminal stabilizing costs  \cite{Mayne2013} or as explicit stabilizing constraints \cite{DavidTAC}. Approaches for terminal cost design include, for example, nonlinear LQR design \cite{Raimondo2009}, computation of CLFs using Taylor series approximations of nonlinear dynamics and sum--of--squares techniques \cite{Lucia2015}, or computation of finite--step, non--monotone CLFs (i.e., CLFs that are required to decrease after a finite number of discrete--time steps instead of at every time step) as terminal costs \cite{LazarNMPC18}. Approaches that use control Lyapunov functions explicitly as a constraint within receding horizon control (or model predictive control), are typically referred to as Lyapunov--based MPC, see for example \cite{DavidTAC} and the references therein. Within this approach, flexible, non--monotone control Lyapunov functions were employed in stabilizing receding horizon control in \cite{Lazar2009} and finite--step, non--monotone CLFs \cite{Nikos_FTCLF} were employed in online optimization based stabilizing control in \cite{Noroozi2020}. Compared to standard CLFs, flexible and finite--step (or finite--time) CLFs offer more freedom in the choice of candidate CLFs, but the choice of certain parameters is non--trivial (e.g., $\lambda$--dynamics or $M$--step value).

Approaches that combine dissipation inequalities \cite{Willems2007, Ebenbauer2009} with receding horizon control (or model predictive control) in the spirit of Lyapunov--based MPC have also been proposed in a number of works as follows. A first instance of receding horizon control based on a specific, $\mathcal{H}_\infty$--type of dissipation inequality for discrete--time linear systems, can be found in \cite{Chen2006}. MPC algorithms with passivity--type constraints, which correspond to dissipation inequalities with a specific supply function, were developed in \cite{Raff2007} for continuous--time nonlinear systems and in \cite{Brogi2008} for discrete--time hybrid systems. The passivity based nonlinear MPC formulation of \cite{Raff2007} was extended to distributed control of interconnected continuous--time nonlinear systems in \cite{Varutti2012}, while also using more general \emph{QSR}--type of supply functions \cite{Willems2007}. In \cite{LazarDiss2012} a robustly stabilizing receding horizon controller was developed for discrete--time nonlinear systems based on parameterized input--to--state stability--type of dissipation inequalities. The framework of \cite{Raff2007} was reconsidered in \cite{Yan2016}, where stabilization of isolated continuous--time nonlinear systems using dissipation inequalities with \emph{QSR}--type of supply functions was addressed. More recently, \cite{Tran2018} developed a LMI--based receding horizon control framework for stabilization of discrete--time linear and input affine nonlinear systems based on dissipation inequalities with \emph{QSR}--type of supply functions.

Within the framework of economic MPC, \cite{DongAngeli2017} introduced the notion of a control storage function for discrete-time nonlinear systems in order to design stabilizing terminal ingredients. Specific, zero--sum supply functions related to the economic MPC stage cost were designed therein such that the summation of two control storage functions yields a classical control Lyapunov function. 

In this paper we revisit the stabilization problem for discrete--time nonlinear systems using a new approach, based on dissipation inequalities. Firstly, we provide a result applicable to general storage/supply functions, which proves that dissipativity with a positive definite storage function and a cyclically negative supply function implies asymptotic Lyapunov stability. Compared to \cite{DongAngeli2017} and classical Lyapunov methods, we establish asymptotic Lyapunov stability in a less conservative way, without using classical Lyapunov functions. Secondly, we provide a Massera--type construction of control dissipation (or storage) functions (CDFs) and we design stabilizing receding horizon controllers by minimizing CDFs subject to a cyclically negative supply condition. The developed framework allows usage of arbitrary kernel (positive definite) functions for generating candidate CDFs and corresponding supply functions. The effectiveness of stabilizing controllers based on CDFs is demonstrated in stabilization of interconnected synchronous generators with nonlinear coupling and state/input constraints.


\section{Inferring stability from dissipativity}
\label{sec2}
Let $\Rset$, $\Rset_+$ and $\Nset$ denote the field of real numbers, the set of non--negative reals and the set of natural numbers, respectively. For a vector $x\in\Rset^n$, $\|x\|$ denotes an arbitrary $p$-norm, $p\in\Nset_{\geq 1}\cup\infty$. A function $\varphi:\Rset_+\rightarrow\Rset_+$ \emph{belongs to class $\cK$} if it is continuous, strictly increasing and $\varphi(0)=0$. A function $\varphi:\Rset_+\rightarrow\Rset_+$ \emph{belongs to class $\cK_\infty$} if $\varphi\in\cK$ and $\lim_{s\rightarrow\infty}\varphi(s)=\infty$.

Consider a discrete--time dynamical system
\begin{equation}
\label{eq:2.1}
\begin{split}
x(k+1)&=f(x(k),u(k)), \quad k\in\Nset,
\end{split}
\end{equation}
where $f:\Rset^n\times\Rset^m\rightarrow\Rset^n$ is a suitable function that is zero at zero. We assume that the origin is a stabilizable equilibrium for \eqref{eq:2.1}. The system variables are constrained to compact sets with the origin in their interior, i.e. $(x,u)\in\Xset\times\Uset$. We assume that $\Xset$ is a constrained control invariant set, i.e., for all $x\in\Xset$, there exists a $u:=\kappa(x)\in\Uset$ with $\kappa(0)=0$ such that $f(x,u)\in\Xset$. 
For brevity, we refer to \cite{LazarPhD2006} for the definition of  asymptotic Lyapunov stability in $\Xset$.

Consider the following discrete--time dissipation inequality
\begin{equation}
\label{eq:2.2din}
V(x(k+1))-V(x(k))\leq s(x(k)),\quad\forall k\in\Nset,
\end{equation}
where $V:\Rset^n\rightarrow \Rset$ is a storage function and $s:\Rset^n\rightarrow\Rset$ is a supply function.

The usual requirements for inferring stability from a dissipation inequality \cite{Ebenbauer2009, Gielen2015} are that the storage function $V$ is upper and lower bounded by $\cK_\infty$ functions (which ensures that $V$ is positive definite and radially unbounded) and $s(x):=-\alpha(\|x\|)$ for some $\alpha\in\cK_\infty$. Indeed, under these assumptions \eqref{eq:2.2din} corresponds to Lyapunov's inequality.

In what follows we relax the latter requirement and provide a less conservative condition on the supply function $s$ for inferring asymptotic stability from dissipativity. To this end, the following assumptions are instrumental.
\begin{assumption}
	\label{assum:boundedness}
	\emph{Controlled $\cK$--boundedness:} For the systems dynamics $f(\cdot,\cdot)$ and the state--feedback control law $u(k)=\kappa(x(k))$ that renders $\Xset$ controlled invariant it holds that $\|f(x,\kappa(x))\|\leq\sigma(\|x\|)$ for all $x\in\Xset$ and some $\sigma\in\cK$.
\end{assumption}

The above $\cK$--boundedness property, originally defined in \cite{Nikos_FTCLF} (see Definition~1 therein), is not conservative as it does not imply continuity of $f(\cdot,\cdot)$ (except at the origin, which is a necessary condition for Lyapunov stability). Also, it is worth to mention that $\cK$--boundedness is a necessary condition for $\cK\cL$--stability, as shown in \cite{Gielen2015}.

\begin{assumption}
\label{assum:posdef}
The storage function $V$ satisfies
\begin{equation}
\label{eq:2.3vbounds}
\begin{split}
\alpha_1(\|x\|)\leq V(x)\leq\alpha_2(\|x\|),\quad \forall x\in\Xset,
\end{split}
\end{equation}
for some $\alpha_1,\alpha_2\in\cK_\infty$.
\end{assumption}
The above assumption on the storage function $V$ is used in this paper to ensure stability of \eqref{eq:2.1} via the dissipation inequality \eqref{eq:2.2din}. If the purpose is to ensure dissipative behavior in general, this assumption is not necessary. 
\begin{theorem}
\label{thm:dissip_implies_stability}
Let Assumption~\ref{assum:boundedness} and Assumption~\ref{assum:posdef} hold and let $\alpha_s\in\cK_\infty$. Suppose that $\{u(k)\}_{k\in\Nset}=\{\kappa(x(k))\}_{k\in\Nset}$ is such that the dissipation inequality \eqref{eq:2.2din} holds for all $x(0)\in\Xset$ and all $k\in\Nset$ and $(x(k),u(k))\in \Xset\times\Uset$ for all $k\in\Nset$. Furthermore, suppose that there exists a $M\in\Nset_{\geq 1}$ such that for all $x(0)\in\Xset$ it holds that
\begin{equation}
\label{eq:2.5neg_storage}
\sum_{i=0}^{M-1}s(x(k+i))\leq -\alpha_s(\|x(k)\|),\quad\forall k\in\Nset.
\end{equation}
Then the origin of system \eqref{eq:2.1} in closed--loop with $u(k)=\kappa(x(k))$ is asymptotically Lyapunov stable in  $\Xset$.
\end{theorem}
\begin{proof}
First, we will prove Lyapunov stability.  Adding up inequality \eqref{eq:2.2din} from $k$ to $k+M-1$ and using \eqref{eq:2.5neg_storage} yields that for all $k\in\Nset$, it holds that
\[V(x(k+M))\leq V(x(k))-\alpha_s(\|x(k)\|)\leq V(x(k)).\]
Next, observe that for any $k\in\Nset$ we can write $k=jM+i$ for some $j\in\Nset$ anf $i\in\Nset_{[0,M-1]}$.
Then the above inequality can be rewritten as
\begin{equation}
\label{eq:2Vj}
\begin{split}
V(x((j+1)M+i))\leq V(x(jM+i)),\quad\forall j\in\Nset,
\end{split}
\end{equation}
for all $i\in\Nset_{[0,M-1]}$. Hence, by Assumption~\ref{assum:posdef} we have that for all $x(0)\in\Xset$ and $k\in\Nset$ there exists an $i\in\Nset_{[0,M-1]}$ such that
\[\alpha_1(\|x(k)\|)\leq V(x(k))\leq\ldots\leq V(x(i))\leq\alpha_2(\|x(i)\|).\]
Then, by Assumption~\ref{assum:boundedness}, we have that
\[
\begin{split}
\alpha_2(\|x(i)\|)&=\alpha_2(\|f(x(i-1),\kappa(x(i-1)))\|)\\
&\leq\alpha_2\circ\sigma(\|x(i-1)\|)\\
&\leq \alpha_2\circ\sigma^i(\|x(0)\|),\,\, \forall i\in\Nset_{[0,M-1]}.
\end{split}
\]
Define $\sigma_i(c):=\sigma^i(c)\in\cK$ for all $i\in\Nset_{[0,M-1]}$ and let $\bar\sigma:=\max\{\sigma_1,\ldots,\sigma_{M-1},\text{id}\}\in\cK$. Then, there exists a $\nu\in\cK_\infty$ such that $\bar\sigma(\|x\|)\leq\nu(\|x\|)$ for all $x\in\Xset$ and it holds that for all $x(0)\in\Xset$ and $k\in\Nset$
\[\|x(k)\|\leq \alpha_1^{-1}\circ\alpha_2\circ\nu(\|x(0)\|).\]
Hence, for any $\varepsilon>0$ we can choose $\delta\leq\nu^{-1}\circ\alpha_2^{-1}\circ\alpha_1(\varepsilon)$ such that $\|x(0)\|\leq \delta$ implies $\|x(k)\|\leq\varepsilon$ for all $k\in\Nset$.

Next, we prove asymptotic convergence. As already shown, it holds that
\begin{equation}
\label{eq:2Mstep}
V(x(k+M))-V(x(k))\leq -\alpha_s(\|x(k)\|),\quad \forall k\in\Nset.
\end{equation}
By summing up \eqref{eq:2Mstep} for $k=0,1,\ldots,\infty$ we obtain:
\[\begin{split}
&\lim_{j\rightarrow\infty}\left[V(x(jM))+\ldots+V(x(jM+M-1))\right]-\\
&\left[V(x(0))+\ldots+V(x(M-1))\right]\\
&\leq -\sum_{j=0}^\infty[\alpha_s(\|x(jM)\|)+\ldots+\alpha_s(\|x(jM+M-1)\|)].
\end{split}
\]
Since $V$ is lower bounded by zero, inequality \eqref{eq:2Vj} implies $\lim_{j\rightarrow\infty}V(x(jM+i))=V_L^i$ exists for all $i\in\Nset_{[0,M-1]}$. Hence,
\[
\begin{split}
&\sum_{i=0}^{M-1}V(x(i))-\sum_{i=0}^{M-1}V_L^i\\&\geq \sum_{j=0}^\infty[\alpha_s(\|x(jM)\|)+\ldots+\alpha_s(\|x(jM+M-1)\|)]
\end{split}
\]
and the left hand side of the inequality is a finite positive quantity. Then, by a discrete--time correspondent of Barbalat's lemma, see, e.g., \cite[Lemma~2]{Ghaffari2020} ($\alpha_s\in\cK_\infty$ implies that $\alpha_s$ is uniformly continuous on compact sets) we obtain
\[
\begin{split}
\lim_{j\rightarrow\infty}[\alpha_s(\|x(jM)\|)+\ldots+\alpha_s(\|x(jM+M-1)\|)]=0.
\end{split}
\]
Thus, by the properties of $\cK_\infty$--functions it follows that \[\lim_{j\rightarrow\infty}\|x(jM+i)\|=0,\quad \forall i\in\Nset_{[0,M-1]}.\]
Since $\{\|x(jM+i)\|\}_{j\in\Nset}$ with $i\in\Nset_{[0,M-1]}$ are convergent subsequences of $\{\|x(k)\|\}_{k\in\Nset}$ that converge to the same limit (and also include all elements appearing in $\{\|x(k)\|\}_{k\in\Nset}$), it follows that   $\lim_{k\rightarrow\infty}\|x(k)\|=0$, which completes the proof.
\end{proof}

Theorem~\ref{thm:dissip_implies_stability} is a generalization of Lyapunov theorems, which essentially establishes that positive definite storage functions with a cyclically negative supply imply asymptotic Lyapunov stability. For simplicity of exposition in this paper we considered the case of a common cycle period $M$ for all $x\in\Xset$; however, condition \eqref{eq:2.5neg_storage} can be further relaxed by allowing state--dependent cycles, i.e., $M(x)$, under suitable assumptions, which will be addressed in follow up work.
\begin{remark}
Condition \eqref{eq:2.5neg_storage} on the supply $s$ together with the dissipation inequality \eqref{eq:2.2din} imply a strict dissipation inequality, as defined, for example, in \cite{DongAngeli2017}. Indeed, \eqref{eq:2.5neg_storage} implies that
\[V(x(k+1))-V(x(k))\leq -\alpha_s(\|x(k)\|)+\bar{s}(x(k)),\quad\forall k\in\Nset\]
with $\bar{s}(x(k)):=-\sum_{i=1}^{M-1}s(x(k+i))$.
\end{remark}

\section{Control dissipation functions}
\label{sec3}

Next, we define control dissipation functions for constrained discrete--time nonlinear systems. 
\begin{definition}
	\label{ref:CDF}
	A storage function $V$ is a control dissipation function (CDF) in $\Xset$ for system \eqref{eq:2.1} with supply function $s$ if there exists a feedback control law $u(k):=\kappa(x(k))$ with $\kappa(0)=0$ such that $(i)$ $\Xset$ is constrained controlled invariant for $x(k+1)=f(x(k),\kappa(x(k)))$ and $(ii)$ the dissipation inequality \eqref{eq:2.2din} folds for all $x(0)\in\Xset$.
\end{definition}

Note that control dissipation functions recover control Lyapunov functions when the storage function $V$ is positive definite and the supply function $s$ is negative definite. It is worth to recall that a similar notion termed control storage function was previously introduced in \cite{DongAngeli2017}, with the goal to design stabilizing terminal ingredients for economic MPC. Therein, control storage functions were used as a terminal cost. However, asymptotic stability of economic MPC was still established via a standard MPC Lyapunov inequality (i.e., inequality (23) therein) and a standard control Lyapunov function (i.e., $\overline{V}_f(x)$ therein). The latter was obtained as the summation of a control storage function ($V_f(x)$ therein) and a strict storage function ($\lambda(x)$ therein) with zero--sum corresponding supply functions (i.e., $l_{av}^\ast-l(x,u)$ and $l(x,u)-l_{av}^\ast$). Continuity of the storage function was assumed therein, but this is not necessary in the discrete--time setting. 

In what follows we develop a Massera--type construction of candidate CDFs generated by arbitrary positive definite kernel functions $l:\Rset^n\rightarrow\Rset_+$, i.e.,
\begin{equation}
\label{eq:3:Vstorage}
V(x(k)):=\sum_{i=0}^{N-1}l(x(i|k)), \quad k\in\Nset,
\end{equation}
where $N\geq 2$ and $x(i|k)$ is obtained from measured state $x(0|k):=x(k)$, inputs $u(i|k)$, $i\in\Nset_{[0,N-2]}$, with $u(0|k)=u(k)$ and the system model \eqref{eq:2.1}, i.e.,
$x(i+1|k)=f(x(i|k),u(i|k))$, for all $i\in\Nset_{[0,N-2]}$.

Define $\bu(k):=\{u(0|k),\ldots,u(N-2|k)\}$ and consider the following optimization problem.
\begin{problem}[Receding horizon CDF minimization]
\label{OCP1}
\begin{subequations}
\label{eq:3.2}
\begin{align}
\min_{\bu(k)}V(x(k))&=\min_{\bu(k)}\sum_{i=0}^{N-1} l(x(i|k))\label{eq:3.2a}\\
&\text{subject to constraints:}\nonumber\\
x(i+1|k)&=f(x(i|k),u(i|k)),\quad \forall i\in\Nset_{[0,N-2]},\label{eq:3.2b}\\
(x(i+1|k)&,u(i|k))\in\Xset\times\Uset,\quad \forall i\in\Nset_{[0,N-2]}.\label{eq:3.2c}
\end{align}
\end{subequations}
\end{problem}
Let $\bu^\ast(k):=\{u^\ast(i|k)\}_{i\in\Nset_{[0,N-2]}}$ and $\bx^\ast(k):=\{x^\ast(i|k)\}_{i\in\Nset_{[1,N-1]}}$ denote optimal input and state trajectories obtained by solving Problem~\ref{OCP1} at time $k$. Since we assume that $\Xset$ is a constrained control invariant set for system \eqref{eq:2.1}, Problem~\ref{OCP1} is recursively feasible in $\Xset$. Otherwise, the following results hold for the largest constrained control invariant subset of $\Xset$. 

Let $\bu^s(k):=\{u^\ast(1|k),\ldots,u^\ast(N-2|k),u^s(N-1|k)\}\in\Uset^{N-1}$ denote an admissible shifted input trajectory constructed at time $k$ from the optimal input trajectory $\bu^\ast(k)$ computed at time $k$, where $u^s(N-1|k)$ is to be determined.
\begin{theorem}
\label{thm:MPCdissip}
Suppose that $\Xset$ is a constrained control invariant set for system \eqref{eq:2.1}. Given the optimal trajectories $\bu^\ast(k)$ and $\bx^\ast(k)$ of Problem~\ref{OCP1} at time $k$ for current state $x(k)$, let $x^s(N|k):=f(x^\ast(N-1|k),u^s(N-1|k))$ and define the supply function
\begin{equation}
\label{eq:3:sup_MPC}
\begin{split}
s(x(k))&:=l(x^s(N|k))-l(x(k)),\quad k\in\Nset.\\
\end{split}
\end{equation}
Then $V^\ast(x(k)):=\sum_{i=0}^{N-1}l(x^\ast(i|k))$ is a control dissipation function with supply function $s$ for system \eqref{eq:2.1} in closed--loop with $u(k):=u^\ast(0|k)$ obtained by solving Problem~\ref{OCP1}.
\end{theorem}
\begin{proof}
First, notice that the shifted input trajectory $\bu^s(k)$ and corresponding shifted state trajectory are feasible solutions of Problem~\ref{OCP1} at time $k+1$ because $x(k+1)=x^\ast(1|k)\in\Xset$ and there exists a $u^s(N-1|k)\in\Uset$ such that $x^s(N|k)=f(x^\ast(N-1|k),u^s(N-1|k))\in\Xset$, due to constrained control invariance of $\Xset$. Then, by optimality at time $k\in\Nset$, it holds that
\[
\begin{split}
V^\ast(x(k+1))&-V^\ast(x(k))\\&= \sum_{i=0}^{N-1}l(x^\ast(i|k+1))-\sum_{i=0}^{N-1}l(x^\ast(i|k))\\
&\leq l(x^s(N|k))+\sum_{i=0}^{N-2}l(x^\ast(i+1|k))\\
&-\sum_{i=1}^{N-1}l(x^\ast(i|k))-l(x^\ast(0|k))\\
&=l(x^s(N|k))-l(x(k))=s(x(k)).
\end{split}
\]
Hence, $V^\ast$ is a CDF with supply function $s$ defined in \eqref{eq:3:sup_MPC} for system \eqref{eq:2.1} in closed--loop with $u(k):=u^\ast(0|k)$ obtained by solving Problem~\ref{OCP1}.
\end{proof}

Problem~\ref{OCP1} can be used to ensure a desired closed--loop dissipative behavior  prescribed by specific choices of storage and supply functions. 

\subsection{CDFs with a cyclically negative supply}

Next, we will combine the result of Theorem~\ref{thm:MPCdissip} with the result of Theorem~\ref{thm:dissip_implies_stability} to design stabilizing receding horizon controllers for discrete--time nonlinear systems.
\begin{corollary}
\label{stab:MPC}
Let $l$ satisfy inequality \eqref{eq:2.3vbounds} for some $\alpha_1, \alpha_2\in\cK_\infty$. Assume that $\Xset$ is a constrained control invariant set for system \eqref{eq:2.1}. Suppose that Assumption~\ref{assum:boundedness} holds and there exists a $\rho\in\cK_\infty$, $\rho<\text{id}$, such that the supply function $s$ defined in \eqref{eq:3:sup_MPC} satisfies for all $x(0)\in\Xset$
\begin{equation}
\label{eq3:MPC:slessthanl}
\sum_{i=0}^{M-1}s(x(k+i))\leq -\rho\circ l(x(k)),\quad\forall k\in\Nset.
\end{equation}
Then for system \eqref{eq:2.1} in closed--loop with $u(k):=u^\ast(0|k)$ obtained by solving Problem~\ref{OCP1} it holds that $V^\ast$ is a control dissipation function in $\Xset$ with supply function $s$ and the origin is asymptotically Lyapunov stable in $\Xset$.
\end{corollary}
\begin{proof}
From \eqref{eq3:MPC:slessthanl} and the fact that $l$ satisfies inequality \eqref{eq:2.3vbounds} we obtain that the supply function $s$ satisfies inequality \eqref{eq:2.5neg_storage} with $\alpha_s(c):=\rho\circ\alpha_{1}(c)\in\cK_\infty$. Moreover, it holds that \[V(x(k))\geq l(x(k))\geq \alpha_1(\|x(k)\|),\quad \forall x(k)\in\Xset,\]
and that there exists $\bar{\alpha}_2\in\cK_\infty$ such that $V(x(k))\leq \bar{\alpha}_{2}(\|x(k)\|)$ for all $x(k)\in\Xset$, due to the upper bound on $l$ and the $\cK$--boundedness property of $f(\cdot,\cdot)$. Hence, $V$ satisfies Assumption~\ref{assum:posdef}. The claim then follows from Theorem~\ref{thm:dissip_implies_stability}.
\end{proof}

Problem~\ref{OCP1} recovers the discrete--time correspondent of Sontag's formula as a particular case in the ideal situation when there is no plant--model mismatch and no disturbance. Indeed, for $N=2$ and given $x(k)$, minimization of $V(x(k))$ implies minimization of $l(x(k+1))=l(x(1|k))$, case in which $l$ becomes a candidate CLF. However, if $N>2$, $l$ is no longer required to be a CLF in $\Xset$, and can be taken as any positive definite function in principle. A sensible choice would be to take $l$ as a local CLF, e.g., computed for the linearization of \eqref{eq:2.1} at the origin, as this choice is non--conservative as the state converges closer to the origin.

Next, we develop a framework for verifying or enforcing the cyclically negative supply condition \eqref{eq3:MPC:slessthanl} online. To this end recall the optimal sequences $\bx^\ast(k), \bu^\ast(k)$ computed at time $k$ by solving Problem~\ref{OCP1} and consider the following horizon--1 optimization problem.
\begin{problem}[Dynamic supply construction]
\label{OCP1:1}
\begin{subequations}
\label{eq3:MPChorizon1}
\begin{align}
\min_{u^s(N-1|k)}&l(x^s(N|k))\label{eq3:MPChorizon1a}\\
&\text{subject to constraints:}\nonumber\\
x^s(N|k)&=f(x^\ast(N-1|k),u^s(N-1|k)),\label{eq3:MPChorizon1b}\\
(x^s(N|k)&,u^s(N-1|k))\in\Xset\times\Uset.\label{eq3:MPChorizon1c}
\end{align}
\end{subequations}
\end{problem}
The resulting optimal pair $(x^{s\ast}(N|k),u^{s\ast}(N-1|k))$ can then be used to compute
\[s^\ast(x(k)):=l(x^{s\ast}(N|k))-l(x(k)).\]
Then, it can be determined online for $k\geq M-1$ if the cyclically negative supply condition holds, i.e.
\begin{equation}
\label{eq:3onlinecheck}
\begin{split}
\sum_{i=-M+1}^{0}&s^\ast(x(k+i))\leq -\rho\circ l(x(k-M+1)),
\end{split}
\end{equation}
where the value of $M$ can be increased as $k$ increases. The additional computational cost of solving the horizon--1 Problem~\ref{OCP1:1} is negligible compared to solving Problem~\ref{OCP1}. Indeed, if $l$ is taken as a quadratic positive definite function and the dynamics $f$ is input affine, then Problem~\ref{OCP1:1} reduces to a horizon--1 QP problem, which can be solved explicitly \cite{Cimini2017}. Explicit solutions to Problem~\ref{OCP1:1} can also be obtained if this is a nonlinear programming problem, e.g., by using the framework in \cite{Johansen2004}.
\begin{remark}
At first, for a fixed $N\geq 2$, the value of $M$ can be dynamically increased online, if the cyclically negative supply condition does not hold. If a certain maximally allowed $M$ is reached, then the value of $N$ must be increased. This yields a mechanism for learning along closed--loop trajectories how to adapt $N$ to improve the CDF candidate $V$ in terms of inducing a stabilizing behavior, based on the accumulated supply over time.
\end{remark}

Given any positive definite kernel function $l$, $N$ can be regarded as a tuning parameter of the candidate CDF $V$ and $M$ can be regarded as a tuning parameter of the corresponding cyclically negative supply condition. In general, a larger value for $N$ will allow for the latter to be satisfied for a smaller value of $M$, up to the point when $V$ becomes a control Lyapunov function and we can use $M=1$.

Alternatively, Problem~\ref{OCP1} and Problem~\ref{OCP1:1} can be combined into a single problem in the spirit of Lyapunov--based MPC. To this end, consider the accumulated supply available at time $k\geq M$ to realize the cyclically negative supply condition, i.e.
\[\Gamma(k,M):=-\sum_{i=-M+1}^{-1}s(x(k+i))-\rho\circ l(x(k-M+1))\]
and let $\bu(k):=\{u(0|k),\ldots,u(N-1|k)\}$.
\begin{problem}[CDF minimization + supply constraint]
\label{OCP2}
\begin{subequations}
\label{eq:3.2n}
\begin{align}
\min_{\bu(k)}V(x(k))&=\min_{\bu(k)}\sum_{i=0}^{N-1} l(x(i|k))\label{eq:3.2na}\\
&\text{subject to constraints:}\nonumber\\
x(i+1|k)&=f(x(i|k),u(i|k)),\quad \forall i\in\Nset_{[0,N-1]},\label{eq:3.2nb}\\
(x(i+1|k)&,u(i|k))\in\Xset\times\Uset,\quad \forall i\in\Nset_{[0,N-1]},\label{eq:3.2nc}\\
s(x(k))=& l(x(N|k))-l(x(k))\leq \Gamma(k,M),\quad \forall k\geq M.\label{eq:3.2nd}
\end{align}
\end{subequations}
\end{problem}
If there exists a $u^s(N|k)\in\Uset$ such that the shifted sequence
\[\bu^s(k):=\{u^\ast(1|k),\ldots,u^\ast(N-1|k),u^s(N|k)\}\]
is a feasible solution of Problem~\ref{OCP2} at time $k+1$, then closed--loop asymptotic stability in $\Xset$ follows from Corollary~\ref{stab:MPC}. Further research will deal with developing a priori verifiable conditions for existence of such shifted sequences for Problem~\ref{OCP2}.

\subsection{CDFs with cyclically negative supply versus finite--step control Lyapunov functions}
Next, we analyze the relation between CDFs with a cyclically negative supply function and finite--step CLFs, and corresponding stabilizing controllers. To this end, we recall the definition of finite--step (or finite--time) CLFs, originally defined in \cite{Nikos_FTCLF} (see Definition~3 therein) and more recently analyzed in \cite{Noroozi2020}. Let $M\in\Nset$, $M\geq 1$ and consider the inequality
\begin{equation}
	\label{eq:fsCLF}
	V(x(k+M))-V(x(k))\leq -\nu(\|x(k)\|),\quad k\in\Nset,
\end{equation}
where $\nu\in\cK_\infty$ and $x(k+M)$ is a solution of system \eqref{eq:2.1}.
\begin{definition}
	\label{ref:fsCLF}
	A function $V$ that satisfies Assumption~\ref{assum:posdef} is a finite--step control Lyapunov function for system \eqref{eq:2.1} in $\Xset$ if there exists a state--feedback control law $u(k):=\kappa(x(k))$ with $\kappa(0)=0$ such that $(i)$ $\Xset$ is constrained controlled invariant for $x(k+1)=f(x(k),\kappa(x(k)))$ and $(ii)$ the inequality \eqref{eq:fsCLF} holds for all $x(0)\in\Xset$.
\end{definition}
\begin{proposition}
	\label{prop:equiv}
	A function $V$ satisfying Assumption~\ref{assum:posdef} is a CDF for system \eqref{eq:2.1} in $\Xset$ with a supply function $s$ satisfying \eqref{eq:2.5neg_storage} if and only if $V$ is a finite--step CLF for system \eqref{eq:2.1} in $\Xset$.
\end{proposition}
\begin{proof}
	Suppose that $V$ is a CDF for system \eqref{eq:2.1} in $\Xset$ with supply function $s$.  Then, by adding up inequality \eqref{eq:2.2din} and using inequality \eqref{eq:2.5neg_storage} we obtain:
	\[V(x(k+M))-V(x(k))\leq \sum_{i=0}^{M-1}s(x(k+i))\leq -\alpha_s(\|x(k)\|).\]
	Hence, $V$ is a finite--step CLF for system \eqref{eq:2.1} in $\Xset$ with $\nu(c):=\alpha_s(c)\in\cK_\infty$.
	
	Conversely, suppose that system \eqref{eq:2.1} admits a finite--step CLF $V$ in $\Xset$. Define $s(x(k)):=V(x(k+1))-V(x(k))$, which implies $V$ satisfies \eqref{eq:2.2din} for all $k\in\Nset$. Since $V$ satisfies \eqref{eq:fsCLF}, it holds that
	\begin{align*}
		\sum_{i=0}^{M-1}s(x(k+i))&=V(x(k+M))-V(x(k))\\
		&\leq-\nu(\|x(k)\|),\quad\forall k\in\Nset.
	\end{align*}
	Hence, $s$ satisfies \eqref{eq:2.5neg_storage} and $V$ is a CDF for system \eqref{eq:2.1} in $\Xset$ with a cyclically negative supply function.
\end{proof}

The above result shows that CDFs with a cyclically negative supply function provide an alternative characterization of non--monotonic, finite--step CLFs.

Next, we show that stabilizing receding horizon based on CDFs  presents certain advantages. Indeed, let us analyze the relation between stabilizing receding horizon control based on CDFs as in Problem~\ref{OCP2} and stabilizing finite--horizon optimal control based on finite--step CLFs as in problem (OCP-1) in \cite{Noroozi2020}. Firstly, observe that problem (OCP-1) in \cite{Noroozi2020} is a particular instance of Problem~\ref{OCP2} corresponding to the case when the kernel function $l$ is a $N$--step CLF for system \eqref{eq:2.1} and $M=1$. Secondly, \cite{Noroozi2020} uses the converse Lyapunov theorem approach of \cite{Geiselhart2014} to establish closed--loop asymptotic stability, which in turn requires that $\sum_{i=0}^{N-1}l(x(k))$ (with $x(k)$ the actual closed--loop state trajectory) is a standard CLF for system \eqref{eq:2.1}. As a result, in \cite{Noroozi2020} closed--loop asymptotic stability is guaranteed only if problem (OCP-1) (or Problem~\ref{OCP2} with $M=1$) is solved every $N$ discrete--time steps (i.e., at time $k$, $k+N$, $\ldots$) and the complete sequence $\mathbf{u}^\ast(k)$ is applied in open--loop, as done in finite--horizon optimal control.

In contrast, the closed--loop asymptotic stability result of this paper, i.e., Corollary~\ref{stab:MPC}, uses $V^\ast(x(k))=\sum_{i=0}^{N-1}l(x^\ast(i|k))$ as a control dissipation function. In turn, this allows us to establish closed--loop stability in the case when Problem~\ref{OCP2} (with any $M\geq 1$) is solved at every discrete--time step $k$ in a receding horizon fashion and $u(k)=u^\ast(0|k)$ is applied to the system at every time $k$. Additionally, by increasing $M$, we can reduce the value of $N$, which reduces the computational complexity of the corresponding receding horizon controllers. 

\section{Stabilization of interconnected synchronous generators}
\label{sec4}
To illustrate the developed stabilizing control methodology we consider a simplified model of a synchronous generator with significant oscillatory behavior taken from \cite{Dulau2015} and we construct the overall model of 2 interconnected generators with nonlinear coupling:
\[x(k+1)=Ax(k)+B_1u(k)+B_2\sin(x_1(k)-x_3(k)),\,\, k\in\Nset.\]
Above, $x_1(k)$, $x_3(k)$ are the angle deviations and $x_2(k)$, $x_4(k)$ are the frequency deviations of the first and second generator, respectively. The inputs $u_1(k)$, $u_2(k)$ control the mechanical power of the generators. The matrices are
\[
\begin{split}
A&=\diag\left(\begin{bmatrix}1 &31.4159\\ 0 & 0.999\end{bmatrix}, \begin{bmatrix}1 &31.4159\\ 0 & 0.999\end{bmatrix}\right),\\ B_1&=\diag\left(\begin{bmatrix}0\\ 0.01\end{bmatrix}, \begin{bmatrix}0\\ 0.01\end{bmatrix}\right), \\ B_2&=\begin{bmatrix}0 & -0.005& 0& 0.005\end{bmatrix}^\top,
\end{split}
\]
which correspond to the forward Euler discretization of the dynamics from \cite{Dulau2015} for a sampling period of $0.1s$.

The system is subject to the constraints $\Xset:=\{x\in\Rset^4 \ : \|x\|\leq 10\}$ and $\Uset=\{u\in\Rset^2 \ : \|u\|\leq 5\}$. The goal is to stabilize the states at the zero equilibrium from initial condition $x(0)=\begin{bmatrix}0 & 0.15 & 0 & -0.15\end{bmatrix}^\top$, while satisfying state and input constraints.

We implemented the receding horizon CDF controller based on Problem~\ref{OCP1} and Problem~\ref{OCP1:1} with the following settings and parameters: $l(x):=x^\top Qx$, $Q=\diag([0.1, \,10,\, 0.1,\, 10])$, $N=4$ (results provided in Figure~\ref{fig:3}), $N=9$ (results provided in Figure~\ref{fig:4}) and $\rho(c)=0.99c$. We observe that the negative supply condition \eqref{eq:2.5neg_storage} is already feasible for $M=1$ for $N=9$, but in the case $N=4$, a value of $M=5$ is required. For both cases the generators are successfully stabilized, conform with Corollary~\ref{stab:MPC}.
\begin{figure}
	\centering
	\includegraphics[width=1\columnwidth, height=8.5cm]{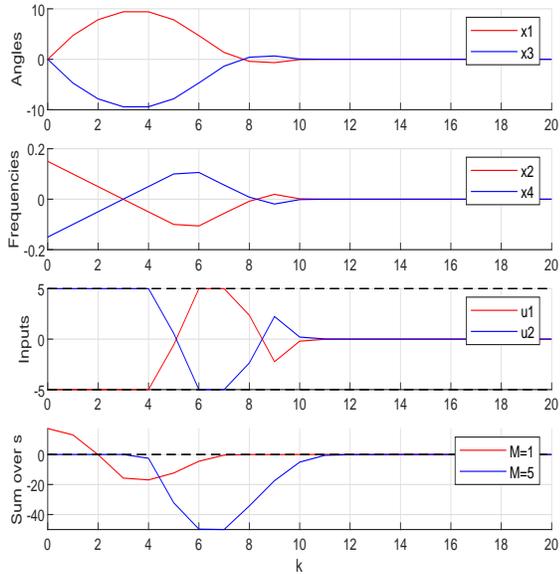}
	\caption{Closed--loop state and input trajectories and accumulated supply over time for Problem~\ref{OCP1} with $N=4$.}
	\label{fig:3}
\end{figure}
\begin{figure}
	\centering
	\includegraphics[width=1\columnwidth, height=8.5cm]{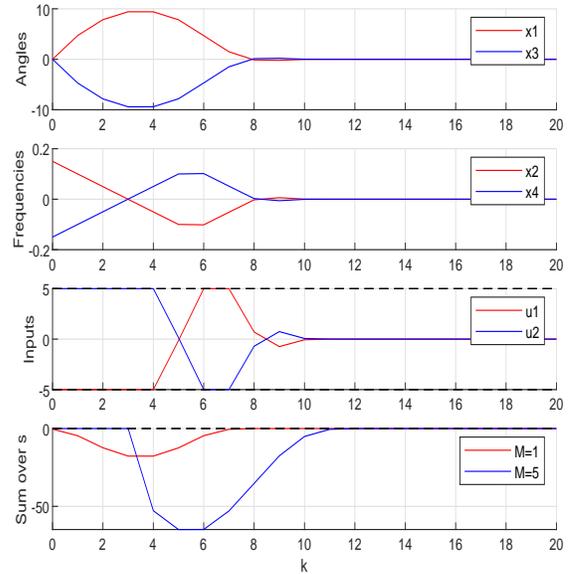}
	\caption{Closed--loop state and input trajectories and accumulated supply over time for Problem~\ref{OCP1} with $N=9$.}
	\label{fig:4}
\end{figure}

It is also worth to point out that $N=4$ is the smallest value which results in converging closed--loop trajectories, while $N=9$ is the smallest value for which \eqref{eq:2.5neg_storage} holds with $M=1$ for the closed--loop trajectories generated via Problem~\ref{OCP1}. Also, simulation results are presented in Figure~\ref{fig:5} for solving Problem~\ref{OCP2} with $M=1$ online, in a receding horizon fashion. In this case, $N=8$ was the smallest value that yielded feasible and converging closed--loop trajectories. 

The above results illustrate the non--conservativeness of stabilizing receding horizon controllers based on CDFs, which achieve convergence for shorter values of $N$, by means of a $M>1$. The values of $N$ and $M$ can be tuned to achieve a trade--off between convergence speed and computational complexity.
\begin{figure}
	\centering
	\includegraphics[width=1\columnwidth, height=8.5cm]{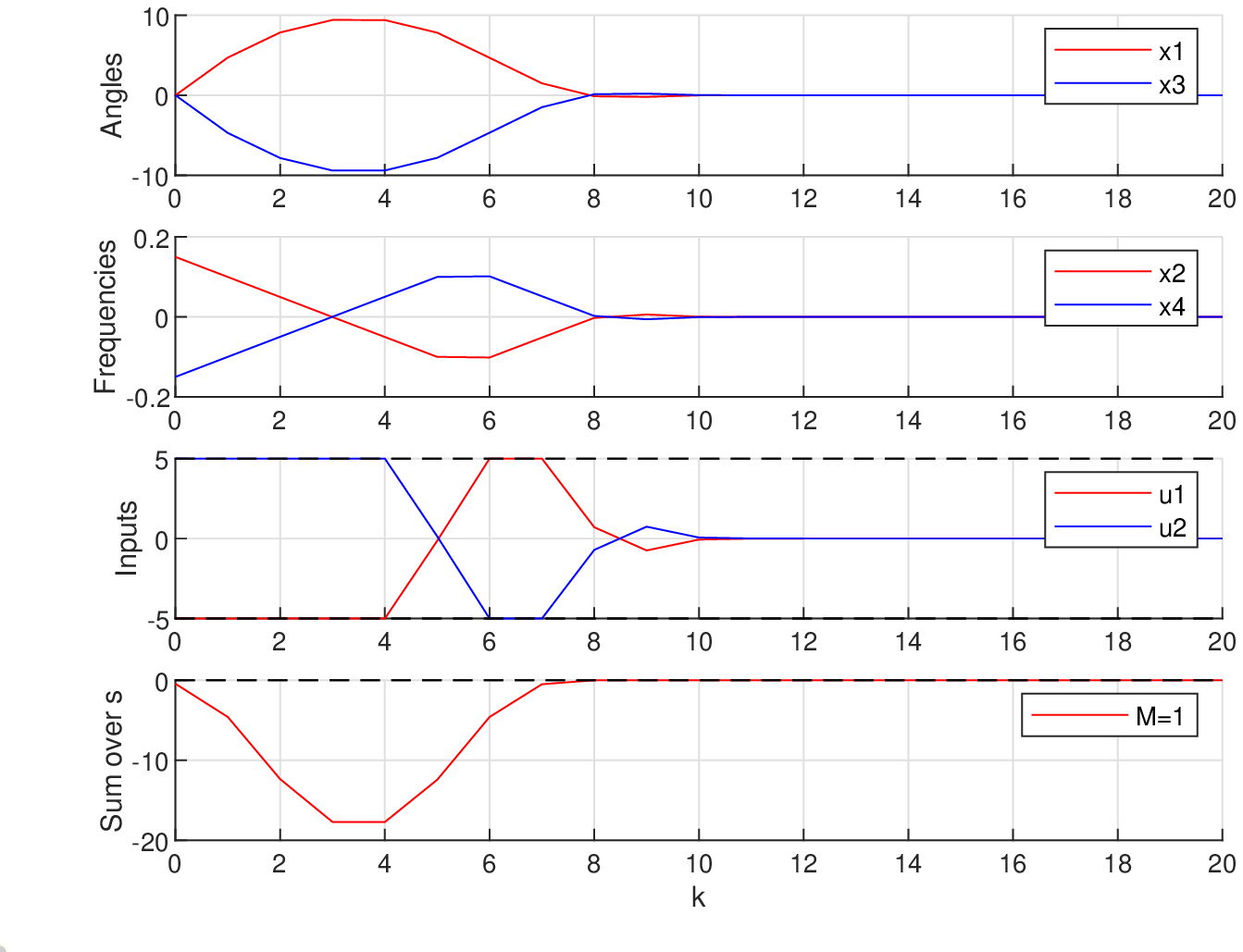}
	\caption{Closed--loop state and input trajectories and accumulated supply over time for Problem~\ref{OCP2} with $N=8$ and $M=1$.}
	\label{fig:5}
\end{figure}

All stabilizing CDF--based controllers have been implemented in Matlab using the YALMIP Toolbox \cite{yalmip} and the \emph{fmincon} Matlab solver.

\section{Conclusions}
\label{sec5}
This paper developed a dissipativity--based framework for synthesis of stabilizing receding horizon controllers for discrete--time nonlinear systems subject to state and input constraints. Key to the proposed approach is a result that establishes that positive definite control dissipation functions and cyclically negative supply functions yield asymptotically stable closed--loop trajectories. Two approaches to the design of stabilizing receding horizon controllers based on minimization of control dissipation functions were developed and demonstrated in stabilization of interconnected synchronous generators with nonlinear coupling. Future work will deal with generalizing the conditions under which CDFs imply asymptotic stability and extensions of CDFs to interconnected nonlinear systems.


\bibliographystyle{IEEEtran}        
\bibliography{Mircea}               

\end{document}